\documentclass[10pt]{amsart}
\usepackage[utf8]{inputenc}
\usepackage{color}
\usepackage{bbm}
\usepackage{amsmath,amssymb}
\usepackage{tikz}
\usetikzlibrary{arrows}
\usepackage{enumerate}
\usepackage{hyperref}
\usepackage{mathrsfs}

\newtheorem{thm}{Theorem}[section]

\newtheorem{lem}[thm]{Lemma}

\theoremstyle{definition}
\newtheorem{defin}[thm]{Definition}
\theoremstyle{remark}
\newtheorem*{rem}{Remark}

\numberwithin{equation}{section}

\newcommand{\lin}{\operatorname{span}}
\newcommand{\supp}{\operatorname{supp}}

\newcommand{\dif}{\,\mathrm{d}}

\newcommand{\charfun}{\ensuremath{\mathbbm 1}}

\allowdisplaybreaks

\begin{document}
\title{Spline characterizations of the Radon-Nikod\'{y}m property}

\author[M. Passenbrunner]{Markus Passenbrunner}
\address{Institute of Analysis, Johannes Kepler University Linz, Austria, 4040 Linz, Altenberger Strasse 69}
\email{markus.passenbrunner@jku.at}

\date{\today}
\begin{abstract}
	We give necessary and sufficient conditions for a Banach space $X$ having the
	Radon-Nikod\'{y}m property in terms of polynomial spline sequences.
\end{abstract}
\maketitle
\section{Introduction and Preliminaries}
The aim of this paper is to prove new characterizations of the Radon-Nikod\'{y}m
property for Banach spaces in terms of polynomial spline sequences 
in the spirit of the corresponding martingale results (see Theorem
\ref{thm:rnp_dentable}).
We thereby continue the line of research about extending martingale results
to also cover (general) spline sequences that is carried out in  \cite{Shadrin2001, PassenbrunnerShadrin2014,
Passenbrunner2014, MuellerPassenbrunner2017, Passenbrunner2017,
KeryanPassenbrunner2017}.
We refer to the book \cite{DiestelUhl1977} by J. Diestel and J.J. Uhl
for basic facts on martingales and vector measures; here, we only give the
necessary notions to define the Radon-Nikod\'{y}m property below.
Let $(\Omega,\mathcal A)$ be a measure space and $X$ a Banach space.
Every $\sigma$-additive map $\nu:\mathcal A\to X$ is called a \emph{vector
measure}.
The \emph{variation} $|\nu|$ of $\nu$ is the set function
\begin{equation*}
	|\nu|(E) = \sup_\pi \sum_{A\in\pi} \|\nu(A)\|_X,
\end{equation*}
where the supremum is taken over all partitions $\pi$ of $E$ into a finite number
of pairwise disjoint members of $\mathcal A$.
If $\nu$ is of bounded variation, i.e., $|\nu|(\Omega)<\infty$, the variation
$|\nu|$ is $\sigma$-additive.
If $\mu:\mathcal A \to [0,\infty)$ is a measure and $\nu:\mathcal A\to X$ is
	a vector measure, $\nu$ is called \emph{$\mu$-continuous} if
	$\lim_{\mu(E)\to 0} \nu(E)=0$ for all $E\in\mathcal A$.
In the following, $L^p_X = L^p_X(\Omega,\mathcal A,\mu)$
will denote the Bochner-Lebesgue space of $p$-integrable Bochner measurable
functions $f:\Omega\to X$ and if $X=\mathbb{R}$, we simply write $L^p$ instead
of $L^p_\mathbb R$.
	\begin{defin}
		A Banach space $X$ has the \emph{Radon-Nikod\'{y}m property
	(RNP)} if
		for every measure space $(\Omega,\mathcal A)$, for every
		positive measure $\mu$ on $(\Omega,\mathcal A)$ and for every
		$\mu$-continuous vector measure $\nu$ of bounded variation,
		there exists a function $f\in L^1_X(\Omega,\mathcal A,\mu)$ such
		that
		\begin{equation*}
			\nu(A) = \int_A f\dif\mu,\qquad A\in\mathcal A.
		\end{equation*}
	\end{defin}

Additionally, recall that a sequence $(f_n)$ in $L^1_X$ is
\emph{uniformly integrable} if the sequence $(\|f_n\|_X)$ is bounded in $L^1$
and, for any $\varepsilon>0$, there exists $\delta>0$ such that
\begin{equation*}
	\mu(A) < \delta \quad \Longrightarrow \quad \sup_n \int_A \|f_n\|_X
	\dif\mu < \varepsilon, \qquad A\in \mathcal A.
\end{equation*}

We have
the following characterization of the Radon-Nikod\'{y}m property in terms of 
martingales, see e.g. {\cite[p. 50]{Pisier2016}}.
\begin{thm}\label{thm:rnp_dentable}
	For any $p\in (1,\infty)$, the following statements about a Banach space
	$X$ are equivalent:
	\begin{enumerate}[(i)]
		\item $X$ has the Radon-Nikod\'{y}m property (RNP),
		\item every $X$-valued martingale bounded in $L^1_X$ converges
			almost surely,
		\item every uniformly integrable $X$-valued martingale converges
			almost surely and in $L^1_X$,
		\item every $X$-valued martingale bounded in $L^p_X$ converges
			almost surely and in $L^p_X$.
	\end{enumerate}
\end{thm}
\begin{rem}
	For the above equivalences, it is enough to consider $X$-valued martingales defined
	on the unit interval with respect to Lebesgue measure and the dyadic
	filtration (cf.
	\cite[p. 54]{Pisier2016}).
\end{rem}

Now, we describe the general framework that allows us to replace properties 
(ii)--(iv) with its spline versions.

\begin{defin}
	A sequence of $\sigma$-algebras $(\mathcal F_n)_{n\geq 0}$ in $[0,1]$ 
	is called an \emph{interval
	filtration} if $(\mathcal F_n)$ is increasing and each  
	$\mathcal F_n$ is generated by a finite partition of $[0,1]$ into
	intervals of positive Lebesgue measure.
\end{defin}
For an interval filtration $(\mathcal F_n)$, we define $\Delta_n := \{\partial A: A
	\text{ atom of }\mathcal F_n\}$ to be the set of all
	endpoints of atoms in $\mathcal F_n$. For a fixed positive integer $k$, set 
	\begin{align*}
		S_n^{(k)} = \{ f\in C^{k-2}[0,1]:\quad &f \text{ is a polynomial of order
		$k$} 
		\text{ on each atom of } \mathcal F_n \},
	\end{align*}
	where $C^n[0,1]$ denotes the space of $n$ times continuously
	differentiable, real valued functions on $[0,1]$ 
	and the order $k$ of a
	polynomial $p$ is related to  the degree $d$ of $p$ by the formula
	$k=d+1$.

	The finite dimensional space $S_n^{(k)}$ admits a very special basis
	$(N_i)$ of non-negative and
	uniformly bounded functions, called
	B-spline basis, that forms a partition of unity, i.e. $\sum_i N_i(t) = 1$ for
	all $t\in [0,1]$, and the support of each $N_i$ consists of the union of
	$k$ neighboring atoms of $\mathcal F_n$. If $n\geq m$ and
	$(N_i),(\tilde{N}_i)$ are the B-spline bases of $S_n^{(k)}$ and
	$S_m^{(k)}$ respectively, we can write each $f\in S_m^{(k)}$ as
	$f= \sum a_i \tilde{N}_i = \sum b_i N_i$ for some coefficients
	$(a_i),(b_i)$ since $S_m^{(k)}\subset S_n^{(k)}$. Those coefficients are
	related to each other in the way that each
	$b_i$ is a  convex combination of the coefficients
	$(a_i)$.
	For more information on spline
	functions, see
	\cite{Schumaker2007}.

	Additionally, we let $P_n^{(k)}$ be the orthogonal projection operator onto
	$S_n^{(k)}$ with
	respect to $L^2[0,1]$ equipped with the Lebesgue measure $|\cdot|$. Each space
	$S_n^{(k)}$
	is finite dimensional and B-Splines are uniformly bounded, therefore, 
	$P_n^{(k)}$ can be extended to $L^1$ and $L^1_X$ satisfying
	$P_n^{(k)}(f\otimes x) = (P_n^{(k)}f) \otimes x$ for all $f\in L^1$ and
	$x\in X$, where $f\otimes x$ denotes the function $t\mapsto f(t)x$.
	Moreover, by $S_{n}^{(k)}\otimes X$, we denote the space $\lin\{f\otimes x : f\in
		S_n^{(k)}, x\in X\}$.
	
	\begin{defin}
		Let $X$ be a Banach space and
		$(f_n)_{n\geq 0}$ be a sequence of functions in $L^1_X$. Then, $(f_n)$
		is an ($X$-valued) \emph{$k$-martingale spline sequence
		adapted to $(\mathcal F_n)$},
		if
			$(\mathcal F_n)$ is an interval filtration
			and
			\begin{equation*}
				P_n^{(k)} f_{n+1} =f_n,\qquad n\geq 0.
			\end{equation*}
	\end{defin}

	This definition resembles the definition of martingales with the
	conditional expectation operator replaced by $P_n^{(k)}$.
	For splines of order $k=1$, i.e. piecewise constant functions,
	the operator $P_n^{(k)}$ even is the conditional
	expectation operator with respect to the $\sigma$-algebra $\mathcal
	F_n$.

	Many of the results that are true for martingales (such as Doob's
	inequality, the martingale convergence theorem or Burkholder's
	inequality) in fact carry over to
	$k$-martingale spline sequences corresponding to an \emph{arbitrary} interval
	filtration as the following two theorems show:
	\begin{thm}\label{thm:splines}
		For any positive integer $k$, any interval filtration
		$(\mathcal F_n)$ and any Banach space $X$, the following
		assertions are true:
		\begin{enumerate}[(i)]
			\item\label{it:splines1} 
				there exists
				a constant $C_k$ only depending on $k$ such that 
				\[
					\sup_n\| P_n^{(k)} : L^1_X \to L^1_X \| \leq
					C_k,
				\]
			\item \label{it:splines2}
				there exists a constant $C_k$ only depending on
				$k$ such that for any $X$-valued $k$-martingale
				spline sequence
				$(f_n)$ and any
				$\lambda>0$, 
				\begin{equation*}
					|\{ \sup_n \|f_n\|_X > \lambda \}| \leq C_k
				\frac{\sup_n\|f_n\|_{L^1_X}}{
				\lambda},
				\end{equation*}

			\item\label{it:splines3} for all $p\in (1,\infty]$ there exists a constant
				$C_{p,k}$ only depending on $p$ and
				$k$ such that for all $X$-valued $k$-martingale
				spline sequence
				$(f_n)$,
				\begin{equation*}
					\big\| \sup_n \|f_n\|_X \big\|_{L^p}
					\leq C_{p,k}
					\sup_n\|f_n\|_{L^p_X},\ 
				\end{equation*}
			\item\label{it:splines4}
				if $X$ has the RNP and $(f_n)$ is an
				$L^1_X$-bounded $k$-martingale spline sequence,
				$(f_n)$ converges
				a.s. to some $L^1_X$-function.
		\end{enumerate}
	\end{thm}
	\eqref{it:splines1} is proved in \cite{Shadrin2001}  and
	\eqref{it:splines2}--\eqref{it:splines4} are proved (effectively) in \cite{PassenbrunnerShadrin2014,
				MuellerPassenbrunner2017}.
				
				\begin{thm}[\cite{Passenbrunner2014}]\label{thm:splines5}
				For all $p\in(1,\infty)$ and all positive 
				integers $k$, scalar-valued $k$-spline-differences
				converge unconditionally in $L^p$, i.e. for all
				$f\in L^p$,
				\begin{equation*}
					\big\|\sum_n \pm (P_n^{(k)} -
					P_{n-1}^{(k)})f \big\|_{L^p}
					\leq C_{p,k} \|f\|_{L^p},
				\end{equation*}
				for some constant $C_{p,k}$ depending only on
				$p$ and $k$.
	\end{thm}
	The martingale version of Theorem
	\ref{thm:splines5} is
	Burkholder's inequality, which precisely holds
	in the vector-valued setting for UMD-spaces $X$ (by
	the definition of UMD-spaces).
	It is an open problem whether Theorem \ref{thm:splines5} holds for
	UMD-valued $k$-martingale spline sequences in this generality, 
	but see  \cite{KamontMuller2006} for a special case.
	For more information on UMD-spaces, see e.g. \cite{Pisier2016}.

	\begin{defin}
		Let $X$ be a Banach space, $(\mathcal F_n)$ an interval filtration and $k$ a
		positive integer. Then, $X$ has the \emph{$((\mathcal
		F_n),k)$-martingale spline convergence property (MSCP)} if 
		all $L^1_X$-bounded $k$-martingale spline sequences adapted to
		$(\mathcal F_n)$ admit a
		limit almost surely.
	\end{defin}

	In this work, we prove the following characterization of the Radon-Nikod\'{y}m
	property in terms of $k$-martingale spline sequences.

\begin{thm}\label{thm:char}
	Let $X$ be a Banach space, $(\mathcal F_n)$ an interval filtration, $k$ a positive integer
	and $V$ the set of all accumulation points of $\cup_n \Delta_n$. Then, 
	$((\mathcal F_n),k)$-$\operatorname{MSCP}$ characterizes RNP if and only if $|V|>0$,
	i.e.,
	\begin{equation*}
		|V|>0\Longleftrightarrow\big( X \text{ has RNP}  \Longleftrightarrow X \text{ has }
		((\mathcal F_n),k)\text{-MSCP}\big).
	\end{equation*}
\end{thm}
\begin{proof}
	If $|V|>0$, it follows from Theorem \ref{thm:splines}\eqref{it:splines4} that RNP
	implies $( (\mathcal F_n), k)$-MSCP for any positive integer $k$ and any
	interval filtration $(\mathcal F_n)$. The reverse implication for
	$|V|>0$ is a consequence of Theorem \ref{thm:general}. We even have that
	if $X$ does not have RNP, we can find a $(\mathcal F_n)$-adapted
	$k$-martingale spline sequence that does not converge at all points
	$t\in E$ for a subset $E\subset V$ with $|E|=|V|$. We simply have to
	choose $E:=\limsup E_n$ with $(E_n)$ being the sets from Theorem
	\ref{thm:general}.

	If $|V|=0$, it is proved in \cite{MuellerPassenbrunner2017} that any
	Banach space $X$ has $((\mathcal F_n),k)$-MSCP.
\end{proof}

We also have the following spline analogue of
  Theorem \ref{thm:rnp_dentable}:
	\begin{thm}\label{thm:rnpspline}
		For any positive integer $k$ and any $p\in (1,\infty)$,
		the following statements about a Banach space $X$ are
		equivalent:
		\begin{enumerate}[(i)]
			\item \label{it:rnp}$X$ has the Radon-Nikod\'{y}m property,
			\item \label{it:L1as}every $X$-valued $k$-martingale
				spline sequence bounded in
				$L^1_X$ converges almost surely,
			\item \label{it:unif}every uniformly integrable
				$X$-valued $k$-martingale spline sequence 
				converges almost surely and in
				$L^1_X$,
			\item \label{it:Lp}every $X$-valued $k$-martingale
				spline sequence bounded
				in $L^p_X$ converges almost surely and in
				$L^p_X$.
		\end{enumerate}
	\end{thm}
	\begin{proof}
		\eqref{it:rnp}$\Rightarrow$\eqref{it:L1as}: Theorem
		\ref{thm:splines}\eqref{it:splines4}. \\
		\eqref{it:L1as}$\Rightarrow$\eqref{it:unif}: clear. \\
		\eqref{it:unif}$\Rightarrow$\eqref{it:Lp}: 
		if $(f_n)$ is a $k$-martingale spline sequence bounded in $L^p_X$ for
		$p>1$, then $(f_n)$ is uniformly integrable, therefore it has a
		limit $f$ (a.s. and $L^1_X$), which, by Fatou's lemma, is also
		contained in $L^p_X$. 
		By Theorem \ref{thm:splines}(iii),  $\sup_n \|f_n\|_X\in L^p$ and 
		we can apply dominated convergence to obtain
		$\|f_n - f\|_{L^p_X}\to 0$. 

		\eqref{it:Lp}$\Rightarrow$\eqref{it:rnp}: follows from Theorem
		\ref{thm:general}.
	\end{proof}

The rest of the article is devoted to the construction of
a suitable non-RNP-valued $k$-martingale spline sequence, adapted to an arbitrary
given filtration $(\mathcal F_n)$, so that the associated martingale
spline differences are separated away from zero on a large set, which, more precisely, takes
the following form:
\begin{thm}\label{thm:general}
Let $X$ be a Banach space without RNP, $(\mathcal F_n)$ an interval filtration,
$V$ the set of all accumulation points of $\cup_n \Delta_n$ and $k$ a positive
integer.

Then,
there exists a positive number $\delta$ such that for all $\eta\in (0,1)$, there
exists an increasing sequence of positive
integers $(m_j)$, an $L_X^\infty$-bounded $k$-martingale spline sequence
$(f_j)_{j\geq 0}$ adapted to $(\mathcal F_{m_j})$ with $f_j \in
S_{m_j}^{(k)}\otimes X$, and a sequence $(E_n)$ of measurable sets
$E_n\subset V$ with $|E_n|\geq
(1-2^{-n}\eta)|V|$  so that for all $n\geq 1$
\begin{equation*}
	\|f_{n}(t) - f_{n-1}(t) \|_X \geq \delta,\qquad   t\in E_n.
\end{equation*}
\end{thm}


We will use the concept of dentable sets to prove  Theorem \ref{thm:general} and
recall its definition:
	\begin{defin}\label{def:dentable}
	Let $X$ be a Banach space. A subset $D\subset X$ is called
	\emph{dentable} if for any $\varepsilon>0$ there is a point $x\in D$
	such that
	\begin{equation*}
		x\notin \overline{\operatorname{conv}}\big(D\setminus
		B(x,\varepsilon)\big),
	\end{equation*}
	where $\overline{\operatorname{conv}}$ denotes the closure of the convex hull and where
	$B(x,\varepsilon)= \{y\in X : \|y-x\| < \varepsilon\}$.
\end{defin}
\begin{rem}[{cf. \cite[p. 138, Theorem 10]{DiestelUhl1977} and \cite[p. 49, Lemma
	2.7]{Pisier2016}}]
	If $D$ is a bounded non-dentable set, then, 
	the closed convex hull $\overline{\operatorname{conv}}(D)$ is also
	bounded and non-dentable. Thus, we may assume that $D$ is convex.
	Moreover, we can as well assume that each $x\in
	D$ can be expressed as a finite convex combination of elements in
	$D\setminus B(x,\delta)$ for some $\delta>0$ since if 
	$D\subset X$ is a convex set such that
		$x\in\overline{\operatorname{conv}}\big(D \setminus
		B(x,\delta)\big)$ for all  $x\in D$,
	then, the enlarged set $\widetilde{D} = D + B(0,\eta)$ is also convex
	and
	satisfies
	\begin{equation*}
		x\in\operatorname{conv}\big(\widetilde{D} \setminus
		B(x,\delta-\eta)\big),\qquad x\in\widetilde{D}.
	\end{equation*}
\end{rem}

The reason why we are able to use the concept of dentability in the proof of
Theorem \ref{thm:general} is the following
geometric characterization of the RNP (see for instance \cite[p. 136]{DiestelUhl1977}).
\begin{thm}\label{thm:dentable}
	For any Banach space $X$ we have that $X$ has the RNP if and only if every bounded subset
	of $X$ is dentable.
\end{thm}

We record the following (special case of the) basic composition formula for determinants (see
for instance \cite[p. 17]{Karlin1968}):
\begin{lem}
	\label{lem:comp}
	Let $(f_i)_{i=1}^n$ and $(g_j)_{j=1}^n$ two sequences of functions in
	$L^2$. Then,
	\begin{align*}
		\det \Big( \int_0^1 f_i(t) &g_j(t)\dif t \Big)_{i,j=1}^n  \\
		&=\int_{0\leq t_1 < \cdots < t_n\leq 1} \det ( f_i(t_\ell)
		)_{i,\ell=1}^n \cdot \det ( g_j(t_\ell))_{j,\ell=1}^n\dif
		(t_1,\ldots,t_n).
	\end{align*}
\end{lem}

We also note the following simple
\begin{lem}
	\label{lem:Gamma}
	Let $I\subset [0,1]$ be an interval and $V$ an arbitrary measurable subset of
	$[0,1]$. 
	Then, for all $\varepsilon_1,\varepsilon_2>0$, there exists a positive
	integer $n$ so that for the decomposition of $I$ into intervals
	$(A_\ell)_{\ell=1}^n$ with $\sup A_\ell \leq \inf A_{\ell+1}$ and
	$n|A_\ell\cap V| = |I\cap V|$ for all $\ell$, the index set 
	$\Gamma = \{ 2 \leq \ell \leq n-1 : \max(|A_{\ell-1}|,
	|A_{\ell}|, |A_{\ell+1}|) \leq \varepsilon_1 \}$
	satisfies 
	\begin{equation*}
	\sum_{\ell\in \Gamma} |A_{\ell} \cap V| \geq (1-\varepsilon_2)
	|I\cap V|.
	\end{equation*}
\end{lem}

\section{Construction of non-convergent spline sequences}
In this section, we prove Theorem \ref{thm:general}.
In order to do that, we begin by fixing an interval filtration $(\mathcal F_n)$, the
corresponding endpoints of atoms $(\Delta_n)$ and a positive integer $k$.
For the space $S_n^{(k)}$, we will suppress the (fixed) index $k$ and write
$S_n$ instead. We will apply the same convention to the corresponding projection
operators $P_n = P_n^{(k)}$.
We also let $V\subset[0,1]$ be the closed set of all accumulation 
points of  $\cup_n \Delta_n$. 

The main step in the proof of Theorem \ref{thm:general} consists of 
an inductive application of the construction of a suitable martingale spline
difference
in the following lemma:

\begin{lem}\label{lem:moments}
	Let  $(x_j)_{j=1}^M$ be in the Banach space $X$, $\bar x\in
	S_N\otimes X$ for some non-negative integer $N$ such that $\bar x =
	\sum_{j=1}^M \alpha_j\otimes x_j$ with
	$\sum_{j=1}^M
	\alpha_j \equiv 1$, $\|x_j\|\leq 1$, $\alpha_j\in S_N$
	having non-negative B-spline coefficients for all $j$ and let $I\subset [0,1]$ be an
	interval so that $|I\cap V| >0$. 
	
	Then, for
	all $\varepsilon\in (0,1)$, there
	exists a positive integer $K$ and a function $g\in S_K\otimes X$ with the properties
	\begin{enumerate}[(i)]
		\item \label{eq:lem2-1}$\int_I t^j g(t)\dif t = 0$ for all $j=0,\ldots,k-1$,
		\item \label{eq:lem2-2}$\supp g\subset \operatorname{int} I$,
		
		\item \label{eq:lem2-3}we have a  splitting of the collection $\mathscr A =
			\{A\subset I: A \text{ is atom in }\mathcal F_K\}$ into
			$\mathscr A_1 \cup \mathscr A_2$  so that  
			\begin{enumerate}[(a)]
				\item 
if the functions $\alpha_j$ are all constant, then \\
on each $J\in \mathscr A_1$, $\bar x + g$
					is constant with a value in
					$\cup_i \{ x_i \}$,
					
					otherwise we still have that \\
					on each $J\in\mathscr A_1$, $\bar{x} + g$
			is constant with a value in
			$\operatorname{conv}\{x_i :
			1\leq i\leq M\}$,

				\item $|\cup_{J\in\mathscr A_1} J\cap V| \geq
					(1-\varepsilon)|I\cap V|$,
				\item on each $J\in \mathscr A_2$,
					$\bar x + g = \sum_\ell
					\lambda_\ell\otimes
					y_\ell$ for some
					functions $\lambda_\ell \in
					S_K$ having non-negative B-spline
					coefficients with $\sum_\ell
					\lambda_\ell\equiv 1$ and
					$y_\ell\in
					\operatorname{conv}\{x_j: 1\leq
					j\leq M\} +
					B(0,\varepsilon)$.
			\end{enumerate}
	\end{enumerate}
\end{lem}

\begin{proof}
	The first step of the construction gives a function $g$ satisfying the
	desired conditions but only having mean zero instead of vanishing moments in
	property \eqref{eq:lem2-1}.
	In the second step, we use this result to construct a function $g$ whose
	moments also vanish.

	\textsc{Step 1:}
			%
	We start with the (simpler) construction of $g$ when the functions
	$\alpha_j$ are not constant and condition \eqref{eq:lem2-3}(a) has the
	form that  on each $J\in\mathscr A_1$, $\bar x + g$ is constant with a
	value in
	$\operatorname{conv}\{x_i : 1\leq i\leq M\}$.
	
	First, decompose $I$ into intervals $(A_\ell)_{\ell=1}^n$ 
	satisfying $n |A_\ell\cap V| = |I\cap V|$ with $\sup A_\ell \leq \inf
	A_{\ell+1}$ and $n\geq 4/\varepsilon$.
	Then, choose $K\geq N$ so large that $A_1,A_2,A_{n-1},A_n$ each
	contains at least $k+1$ atoms of
	$\mathcal F_K$. Denoting by $(N_j)$ the B-spline basis of $S_K$, we
	can write
	\begin{equation*}
		\alpha_\ell \equiv \sum_j \alpha_{\ell,j} N_j, \qquad
		\ell=1,\ldots,M
	\end{equation*}
	for some non-negative coefficients $(\alpha_{\ell,j})$. Define
	\begin{equation*}
		h_\ell \equiv \sum_{j : \cup_{i=2}^{n-1} A_i \cap \supp N_j\neq
	\emptyset} \alpha_{\ell,j} N_j.
	\end{equation*}
	Observe that $\supp h_\ell \subset \operatorname{int} I$ and
	$h_\ell \equiv \alpha_\ell$ on $\cup_{i=2}^{n-1} A_i$. Letting
		$\widetilde{x} = \sum \beta_\ell x_\ell$
	for 
	$	\beta_\ell = \int h_\ell/\big(\sum_j \int
			h_j\big) \in [0,1],$
	we define
	\begin{equation*}
		g := - \sum_{\ell=1}^M h_\ell \otimes x_\ell +
		\Big(\sum_{j=1}^M
		h_j \Big) \otimes \widetilde{x}.
	\end{equation*}
	This is a function of the desired form
	when defining $\mathscr{A}_1 := \{A \subset \cup_{i=2}^{n-1} A_i :
	A\text{ is atom in }\mathcal F_K\}$ and $\mathscr{A}_2 =
	\mathscr{A}\setminus \mathscr{A}_1$ as we will now show by proving
	$\int g=0$  and properties \eqref{eq:lem2-2},\eqref{eq:lem2-3}.
	The fact that $\int g=0$ follows from a simple calculation. Property \eqref{eq:lem2-2} is
	satisfied by the definition of the functions $h_\ell$. Property
	\eqref{eq:lem2-3}(a) follows 
	from the fact
	that $\bar x(t) + g(t) = \tilde{x}\in\operatorname{conv}\{x_j : 1\leq
	j\leq M\}$ for $t\in \cup_{i=2}^{n-1} A_i$ since $h_\ell
	\equiv \alpha_\ell$ on that set for any $\ell=1,\ldots, M$.
	Since $|(A_1\cup A_2\cup A_{n-1}\cup A_n)\cap V|= 4|I\cap V|/n \leq  \varepsilon|I\cap V|$,
	\eqref{eq:lem2-3}(b) also follows from the construction of $\mathscr
	A_1$. Since
	\begin{equation*}
		 \bar x(t)+g(t) = \sum_{\ell=1}^M \big(\alpha_\ell(t) -
		 h_\ell(t)\big) x_\ell +
		\Big(\sum_{j=1}^M h_j(t)\Big) \tilde{x},
	\end{equation*}
	$\tilde{x}\in \operatorname{conv}\{x_j : 1\leq j\leq M\}$, $h_\ell \leq
	\alpha_\ell$ and $\sum_\ell
	\alpha_\ell \equiv 1$, \eqref{eq:lem2-3}(c) is also proved.

	The next step is to construct the desired function $g$ when $\alpha_j$
	are assumed to be constant and \eqref{eq:lem2-3}(a) has the form
that on each $J\in\mathscr A_1$, $\bar x + g$
					is constant with a value in
					$\cup_i \{ x_i \}$.
	Here, the idea is to construct a function of the form
	$g(t) = \sum f_j(t) (x_j - \bar x)$ with $f_j\in S_K$ for some $K$ and $\int f_j \simeq C\alpha_j$
	for all $j$
	and some constant $C$ independent of $j$ to employ the assumption $\sum
	\alpha_j (x_j - \bar x)=0$ implying $\int g = 0$.

	We begin this construction 
	by successively choosing parameters $\varepsilon_3 \ll
	\varepsilon_1 \ll \tilde{\varepsilon}<\varepsilon$
	obeying certain given conditions depending on $\varepsilon$, $\bar x$, 
	$(x_j)$, $(\alpha_j)$, $|I\cap V|$ and $|I|$.

First, set  
		$\tilde{\varepsilon} = \varepsilon |I\cap V|/(3|I|)>0$.
	and
	\begin{equation}
		\label{eq:eps1}
		\varepsilon_1 = \frac{\varepsilon\tilde{\varepsilon}(1-\varepsilon/3) | I\cap V|}{72 M}.
	\end{equation}
	Now, we apply Lemma \ref{lem:Gamma} with the parameters $\varepsilon_1$
	and $\varepsilon_2=\varepsilon/3$ to get a positive integer $n$  and a
	partition
	$(A_{\ell})_{\ell=1}^{n}$ of $I$ consisting of intervals with
	$n|A_\ell\cap V|=|I\cap V|$ for all $\ell=1,\ldots,n$ so that  
	\begin{equation*}
		\Gamma = \{ 2 \leq \ell \leq n-1 : \max(|A_{\ell-1}|,
		|A_{\ell}|, |A_{\ell+1}|) \leq \varepsilon_1 \}
	\end{equation*}
	satisfies
	\begin{equation}\label{eq:applemma}
		\big( 1 - \frac{\varepsilon}{3}\big)|I\cap V|\leq \sum_{\ell\in \Gamma} |A_{\ell} \cap V|.
	\end{equation}
	Finally, we put $\varepsilon_3= \varepsilon_1/(2n)$.

	
	Next, for each $\ell=1,\ldots,n$, we choose a point
	$p_{\ell}\in \operatorname{int} A_{\ell}$ and an integer $K_\ell$ so
	that the intersection of $\operatorname{int} A_{\ell}$ and the $\varepsilon_3$-neighborhood
	$B(p_\ell,\varepsilon_3)$ of $p_{\ell}$ contains at least $k+1$ atoms of
	$\mathcal F_{K_\ell}$
	to the left as well as to the right of $p_{\ell}$. This is possible
	since $|A_\ell\cap V|=|I\cap V|/n$ and $V$ is the set of all
	accumulation points of $\cup_j \Delta_j$. Then set 
	$K=\max_\ell K_\ell$ and we let $u_\ell\in A_{\ell}$
	be the leftmost point of $\Delta_K$ contained in $B(p_\ell,\varepsilon_3)\cap
	\operatorname{int} A_{\ell}$. Similarly, let  $v_\ell\in A_{\ell}$
	be the rightmost point of $\Delta_K$ contained in $B(p_\ell,\varepsilon_3)\cap
	\operatorname{int} A_{\ell}$.
	Next, for $2\leq \ell\leq n-1$, we put $B_{\ell} :=
	(v_{\ell-1}, u_{\ell+1}) \subset A_{\ell-1} \cup A_{\ell} \cup
	A_{\ell+1}$. Observe that the construction of $u_\ell$ and $v_\ell$
	implies that $B_\ell \cap B_{j}=\emptyset$ for all $|\ell - j| \geq 2$.
	Next, let $(N_i)$ be the B-spline basis of the space $S_K$ and let 
	$(\ell(i))_{i= 1}^{L}$ be the increasing sequence of integers 
	so that $\Gamma= \{\ell(i)
	: 1\leq i\leq L\}$ for $L=|\Gamma|\leq n$.
	 We then define the
	set
	\begin{equation*}
		\Lambda( r,s) := \Big\{ j : \supp N_j \cap \Big(\bigcup_{i=r}^s
		B_{\ell(i)}\Big) \neq \emptyset \Big\}
	\end{equation*}
	to consist of those B-spline indices so that the support of the corresponding
	B-spline function intersects the set $\bigcup_{i=r}^s B_{\ell(i)}$.
	Observe that by \eqref{eq:applemma},
	\begin{equation}
		\label{eq:epsest}
		\big(1-\frac{\varepsilon}{3}\big) |I\cap V| \leq \sum_{\ell\in \Gamma}
		|A_{\ell}\cap V| = \Big| \bigcup_{\ell\in \Gamma}
		A_{\ell}\cap V\Big| \leq \Big| \bigcup_{i}
		B_{\ell(i)}\cap V\Big| \leq \Big| \bigcup_{i}
		B_{\ell(i)}\Big|.
	\end{equation}
	Thus, the definition \eqref{eq:eps1} of $\varepsilon_1$  in particular implies 
	\begin{equation}\label{eq:cons_est}
		72\varepsilon_1M 
		\leq  \varepsilon\cdot\tilde{\varepsilon}\cdot \Big|
		\bigcup_i B_{\ell(i)} \Big|.
	\end{equation}

	We continue with defining the functions $(f_j)$ contained in $S_K$ using a stopping time
	construction and 
	first set $j_0=-1$ and $C=(1-\tilde{\varepsilon}/3)\big| \bigcup_{i} B_{\ell(i)}\big|>0$. 
	For $1\leq m\leq M$, if $j_{m-1}$ is already chosen, we
	define $j_m$ to be the smallest integer $\leq L$ so that the function
	\begin{equation}
		\label{eq:cond_def_fm}
		f_m := \sum_{j\in \Lambda(j_{m-1}+2,j_m) } N_j \qquad
		\text{satisfies}\qquad 
		\int f_m(t) \dif t > C\alpha_m.
	\end{equation}
	If no such integer exists, we set $j_m = L$ (however, we will see below
	that for the current choice of parameters, such an integer always exists).
	Additionally, we define
	\begin{equation*}
		f_{M+1} := \sum_{j\in \Lambda(j_{M} + 2, L)} N_j.
	\end{equation*}
	Observe that by the locality of the B-spline basis $(N_i)$, 
	$\supp f_\ell \cap \supp f_m = \emptyset$ for $1\leq \ell <
	m\leq M+1$.
	Based on the collection of functions $(f_m)_{m=1}^{M+1}$, we will define the desired
	function $g$.
	But before we do that, we
	make a few comments about $(f_m)_{m=1}^{M+1}$.

	Note that for $m=1,\ldots, M$, by the minimality of $j_m$,
	\begin{equation*}
		\int \sum_{j\in \Lambda(j_{m-1}+2,j_m -1) } N_j(t)\dif t \leq  C\alpha_m,
	\end{equation*}
	 and therefore, again by the locality of the B-splines $(N_i)$,
	 \begin{equation}\label{eq:alpha_upper}
		\int f_m(t)\dif t \leq C\alpha_m + \int\sum_{j\in
			\Lambda(j_m,j_m) } N_j(t)\dif t \leq C\alpha_m +
			3\varepsilon_1.
	\end{equation}
	Additionally, employing also the definition of $u_\ell$ and $v_\ell$ and
	the fact that the B-splines $(N_i)$ form a partition of unity,
	\begin{equation}
		\label{eq:unionintegral}
		\begin{aligned}
		\Big|\bigcup_{i=j_{m-1}+2}^{j_m}B_{\ell(i)}\Big|&\leq \int f_m(t)\dif t \leq 
		\Big|\bigcup_{i=j_{m-1} + 2}^{j_{m} } (p_{\ell(i)-1},
		p_{\ell(i)+1})\Big| \\
		&\leq\Big|\bigcup_{i=j_{m-1}+2}^{j_m}B_{\ell(i)}\Big| +
			 2n\varepsilon_3.
		 \end{aligned}
	\end{equation}
	Next, we will show
	\begin{equation}\label{eq:B_ell_estimate}
		(1-\tilde{\varepsilon}) \Big|\bigcup_{i} B_{\ell(i)}\Big| \leq
		\Big|\bigcup_{i\leq j_{M}} B_{\ell(i)}\Big| \leq
		(1-\tilde{\varepsilon}/6) \Big|\bigcup_{i} B_{\ell(i)}\Big|.
	\end{equation}
	Indeed, we calculate on the one hand by \eqref{eq:unionintegral} and 
	\eqref{eq:alpha_upper}
	\begin{align*}
		\Big|\bigcup_{i\leq j_{M}} B_{\ell(i)}\Big| &\leq
		\sum_{m=1}^{M}\Big|
		\bigcup_{i=j_{m-1} + 2}^{j_m} B_{\ell(i)}\Big| + \sum_{m=1}^{M}
		|B_{j_m +1}| 
		\leq \sum_{m=1}^{M} \int f_m(t)\dif t + 3\varepsilon_1 M \\
		&\leq \sum_{m=1}^{M} (C\alpha_m + 3\varepsilon_1) +
		3\varepsilon_1 M = C +
		6\varepsilon_1 M
	\end{align*}
	Recalling now that $C= (1-\tilde{\varepsilon}/3) \big| \bigcup_i
	B_{\ell(i)}\big|$ and using \eqref{eq:cons_est} now yields 
	the right hand side of
	\eqref{eq:B_ell_estimate}.

	On the other hand, employing
	\eqref{eq:unionintegral} and \eqref{eq:cond_def_fm}, 
	\begin{align*}
		\Big|\bigcup_{i\leq j_{M}} B_{\ell(i)}\Big| &\geq \sum_{m=1}^{M}
		\Big|\bigcup_{i=j_{m-1} + 2}^{j_m} B_{\ell(i)}\Big| \geq \sum_{m=1}^{M}
		\Big(\int f_m(t)\dif t -  2n\varepsilon_3 \Big)\\
		&\geq C\sum_{m=1}^{M} \alpha_m - 2nM\varepsilon_3 = C -
		2nM\varepsilon_3.
	\end{align*}
	The definition of $C = (1-\tilde{\varepsilon}/3)\big|
	\bigcup_i B_{\ell(i)} \big|$ and $\varepsilon_3 = \varepsilon_1	/(2n)$,
	combined with \eqref{eq:cons_est} gives 
	the left hand inequality in \eqref{eq:B_ell_estimate}.

	The inequality on the right hand side of \eqref{eq:B_ell_estimate},
	combined with  \eqref{eq:cons_est} again,
	allows us to give the following lower estimate of $\int f_{M+1}$:
\begin{equation}
	\label{eq:intfM}
	\begin{aligned}
		\int f_{M+1}(t)\dif t &\geq \Big| \bigcup_{i\geq j_{M}+2} B_{\ell(i)} \Big| \geq 
		\Big| \bigcup_{i> j_{M}} B_{\ell(i)} \Big| - 3\varepsilon_1 
		\geq \frac{\tilde{\varepsilon}}{12} \Big| \bigcup_i B_{\ell(i)}
		\Big|.
	\end{aligned}
\end{equation}
	We are now ready to define the function $g\in S_K\otimes X$ as follows:
	\begin{equation}
		\label{eq:defg}
		g\equiv \sum_{j=1}^{M} f_j\otimes (x_j - \bar x) +
		f_{M+1}\otimes\sum_{j=1}^M
		\beta_j (x_j-\bar x) ,
	\end{equation}
	where 
	\begin{equation}
		\label{eq:defbeta}
		\beta_j = \frac{C\alpha_j - \int f_j(t)\dif t}{\int f_{M+1}(t)\dif
		t}, \qquad 1\leq j\leq M.
	\end{equation}
	We proceed by proving $\int g=0$ and properties \eqref{eq:lem2-2}--\eqref{eq:lem2-3} for $g$:

	The fact that $\int g=0$ follows from  a straightforward calculation using 
	\eqref{eq:defbeta} and the assumption $\sum_{j=1}^M
	\alpha_j (x_j - \bar x)=0$.
	\eqref{eq:lem2-2} follows from $\supp g \subset [p_1,
	p_{n}]\subset \operatorname{int}I$.
Next, observe that  by definition of $g$ and
$f_1,\ldots, f_{M+1}$, on each $\mathcal F_K$-atom contained in the set 
$B:=\cup_{m=1}^{M} \cup_{i=j_{m-1} +2}^{j_m} B_{\ell(i)}$, the function $\bar x
+ g$ is constant with a value in $\cup_i\{x_i\}$.
	Setting $\mathscr A_1= \{A \subset B : A\text{ atom in }\mathcal F_K\}$
	and  $\mathscr A_2 = \mathscr A\setminus \mathscr A_1$ now shows
	\eqref{eq:lem2-3}(a).
	Moreover, by \eqref{eq:eps1}, \eqref{eq:epsest} and
	\eqref{eq:B_ell_estimate},
	\begin{align*}
		\Big| \bigcup_{J\in\mathscr{A}_1} J\cap V \Big| &= \Big| \bigcup_{m=1}^{M}
		\bigcup_{i=j_{m-1} +2}^{j_m} B_{\ell(i)} \cap V\Big| \geq \Big|
		\bigcup_{i\leq j_{M}} B_{\ell(i)} \cap V \Big| -
		3M\varepsilon_1 \\
		&\geq \Big| \bigcup_i B_{\ell(i)}\cap V \Big| - \Big|
		\bigcup_{i>j_{M}} B_{\ell(i)}\Big|  - \frac{\varepsilon |I\cap V|}{24} \\
		&\geq \big( 1 - \frac{\varepsilon}{3} \big) |I\cap V| -
		\tilde{\varepsilon} \Big|\bigcup_i B_{\ell(i)}\Big| -
		\frac{\varepsilon|I\cap V|}{24}.
	\end{align*}
	Since $\tilde{\varepsilon}|\cup_i B_{\ell(i)} |\leq \tilde{\varepsilon}
	|I|\leq \varepsilon|I\cap V|/3$ by definition of $\tilde{\varepsilon}$,
	we conclude $|\cup_{J\in \mathscr A_1} J\cap V| \geq (1-\varepsilon)|I\cap V|$, proving
	also 
	\eqref{eq:lem2-3}(b).
	Next, we note that for $t\in \supp f_j$ with $j\leq
	M$, we have
	\begin{equation*}
		\bar x + g(t) = \bar x + f_j(t) (x_j - \bar x)  =
		f_j(t) x_j + \big(1-f_j(t)\big) \bar x.
	\end{equation*}
	Since $f_j(t)\in [0,1]$ and $\bar x$ is a convex combination of the elements
	$(x_j)$, we get \eqref{eq:lem2-3}(c) in this case.
	If $t\in \supp f_{M+1}$, we calculate
	\begin{equation}
		\label{eq:uniformestimate_g}
		\begin{aligned}
		\bar x + g(t)  &= \bar x +f_{M+1}(t) \sum_{j=1}^M \beta_j (x_j - \bar x)
		\\
		&= \big(1-f_{M+1}(t)\big) \bar x + f_{M+1}(t) \big( \bar x +
		\sum_{j=1}^M \beta_j (x_j - \bar x)\big).
	\end{aligned}
	\end{equation}
	We have by the lower estimate \eqref{eq:intfM} for $\int f_{M+1}$ and by
	\eqref{eq:alpha_upper}
\begin{align*}
	\sum_{j=1}^M |\beta_j| &\leq \frac{12}{\tilde{\varepsilon} \big|\bigcup_i
	B_{\ell(i)}\big|}   \sum_{j\leq M} \Big(\int f_j - C\alpha_j\Big)  
	\leq \frac{12}{\tilde{\varepsilon} \big|\bigcup_i
	B_{\ell(i)}\big|} ( 3\varepsilon_1 M ),
\end{align*}
which, by \eqref{eq:cons_est},
is smaller
than $\varepsilon/2$. Therefore, combining this 
with \eqref{eq:uniformestimate_g} yields property \eqref{eq:lem2-3}(c) for $t\in \supp
f_{M+1}$ 
by setting $\lambda_1=1-f_{M+1}$, $\lambda_2 = f_{M+1}$, $y_1 = \bar x$, $y_2 =
\bar x + \sum_j \beta_j (x_j - \bar x)$.
Thus, we finished Step 1 of constructing a function $g$ with mean zero and
properties \eqref{eq:lem2-2},\eqref{eq:lem2-3}.
The next step is to construct a function $g$ so that additionally all
of its moments up to order $k$ vanish. 

\textsc{Step 2:}
%
	Set $\tilde{\varepsilon} = 1- (1-\varepsilon)^{1/3}>0$.
	We write $a=\inf I$, $b=\sup I$ and choose $c\in I$ so that $R := (c,b)$
	satisfies
	$0<|R \cap V|= \tilde{\varepsilon}|I\cap V|$. Define $L=I\setminus R$.
	Let $(N_i)$ be the B-spline basis of $S_{K_R}$, 
	where we choose the integer $K_R$ so that
	we can select B-spline
	functions $(N_{m_i})_{i=0}^{k-1}$ 
	that $\supp N_{m_i} \subset \operatorname{int} R$ for any $i=0,\ldots,k-1$
	and $\supp N_{m_i}\cap \supp N_{m_j} = \emptyset$ for $i\neq j$.
	We then form the $k\times k$-matrix
	\begin{equation*}
		A  = \Big( \int_R t^i N_{m_j}(t)\dif t \Big)_{i,j=0}^{k-1}.
	\end{equation*}
	The matrix $(t_\ell^i)_{i,\ell=0}^{k-1}$ is a Vandermonde matrix having
	positive determinant for $t_0 < \cdots< t_{k-1}$. Moreover, the matrix
	$(N_{m_j}(t_\ell))_{j,\ell=0}^{k-1}$ is a diagonal matrix having positive
	entries if $t_\ell\in \operatorname{int}\supp N_{m_\ell}$ for
	$\ell=0,\ldots, k-1$. For other choices of $(t_\ell)$, the determinant
	of $(N_{m_j}(t_\ell))_{j,\ell=0}^{k-1}$ vanishes.
	Therefore, Lemma \ref{lem:comp} implies that $\det A\neq 0$ and $A$ is invertible.

	Next, we choose 
	$\varepsilon_1= \tilde{\varepsilon}/\big(k(1+\tilde{\varepsilon}) \| A^{-1} \|_\infty
	|L| \big)$
	and apply Lemma \ref{lem:Gamma} with the parameters $\varepsilon_1$,
	$\varepsilon_2 = \tilde{\varepsilon}$ and the interval $L$ to obtain
	a positive integer $n$ so
	that
	for the partition $(A_\ell)_{\ell=1}^n$ of $L$ with $n|A_\ell \cap V| =
	|L\cap V|$ and $\sup A_{\ell-1} = \inf A_{\ell}$,
	the set
		$\Gamma = \{ 2\leq \ell \leq n-1 : \max(|A_{\ell-1}|,
		|A_{\ell}|, |A_{\ell+1}|) \leq \varepsilon_1 \}$
	satisfies
	\begin{equation*}
		\sum_{\ell\in \Gamma} |A_{\ell} \cap V| \geq
		(1-\tilde{\varepsilon})
		|L\cap V|.
	\end{equation*}
	We now apply the construction of Step 1 on every set $A_{\ell}$, 
	$\ell\in \Gamma$, with the parameters $\bar x$, $(x_j)_{j=1}^M$,
	$(\alpha_j)_{j=1}^M$, $\tilde{\varepsilon}$ 
	to get functions $(g_{\ell})$  with zero mean having 
	properties \eqref{eq:lem2-2},\eqref{eq:lem2-3}
	with $I$ replaced by $A_{\ell}$. On $L$, we define the function 
	\begin{equation*}
		g(t) := \sum_{\ell\in \Gamma} g_\ell(t),\qquad t\in L.
	\end{equation*}
	Let $z_j:= \int_L t^j g(t)\dif t$ for $j=0,\ldots,k-1$. Observe that,
	since $\int_{A_{\ell}} g_\ell(t)\dif t=0$ and $\|g_\ell\|_{L^\infty_X}
	\leq 1+\tilde{\varepsilon}$ by \eqref{eq:lem2-3} 
	 and $|A_\ell|\leq
	\varepsilon_1$, we get for all $j=0,\ldots,k-1$,
	\begin{align*}
		\| z_j \| &= \Big\|  \sum_{\ell\in\Gamma} \int_{A_{\ell}}
		t^jg_\ell(t)\dif t\Big\| = \Big\|  \sum_{\ell\in\Gamma} \int_{A_{\ell}}
		\big(t^j-(\inf A_\ell)^j\big)\cdot g_\ell(t)\dif t\Big\| \\
		&\leq j \sum_{\ell\in\Gamma} |A_{\ell}|\int_{A_{\ell}} \|
		g_\ell(t) \| 
		\dif t \\
		&\leq j \varepsilon_1 (1+\tilde{\varepsilon}) 
		|L| \leq \tilde{\varepsilon}\cdot \|A^{-1}\|_\infty^{-1}.
	\end{align*}
	In order to have $\int_I t^j g(t)\dif t = 0$ for all $j=0,\ldots,k-1$, we want to define $g$ on
	$R=I\setminus L$
	so that
	\begin{equation}\label{eq:condR}
		\int_R t^j g(t)\dif t = - z_j,\qquad j=0,\ldots,k-1.
	\end{equation}
	Assume that $g$ on $R$ is of the form
	\begin{equation*}
		g(t) = \sum_{i=0}^{k-1} N_{m_i}(t)w_i,\qquad t\in R
	\end{equation*}
	for some $(w_i)_{i=0}^{k-1}$ contained in $X$. Then, \eqref{eq:condR} is
	equivalent to 
	\begin{equation*}
		Aw = -z
	\end{equation*}
	by writing $w=(w_0,\ldots,w_{k-1})^T$ and $z=(z_0,\ldots,z_{k-1})^T$.
	Defining $w := -A^{-1} z$ and employing 
	the estimate for $\|z\|_\infty$ above, we obtain
	\begin{equation}\label{eq:esty}
		\|w\|_\infty \leq \|A^{-1}\|_\infty \|z\|_\infty \leq
		\tilde{\varepsilon}.
	\end{equation}

	The definition of $g$ immediately yields properties \eqref{eq:lem2-1},
	\eqref{eq:lem2-2}.
	From the application of the construction in Step 1 to each
	$A_{\ell}$, $\ell\in\Gamma$, we obtained collections $\mathscr
	A_1(\ell)$ of disjoint subintervals of $A_\ell$ that are atoms in $\mathcal
	F_{K_\ell}$ for some positive integer $K_\ell\geq N$ satisfying that $\bar x +
	g_\ell$ is constant on each $J\in \mathscr A_1(\ell)$ taking values in
	$\operatorname{conv}\{ x_i : 1\leq i\leq M \}$ and $|\cup_{J\in\mathscr A_1(\ell)} J\cap V| \geq
	(1-\tilde{\varepsilon}) |A_\ell \cap V|$.
	Let $B:= \cup_\ell\cup_{J\in\mathscr A_1(\ell)} J$ and define
	$\mathscr A_1$ to be the collection $\{J\subset B : J\text{ atom in }\mathcal
	F_K\}$ where $K:= \max(\max_\ell K_\ell, K_R	)$ and define $\mathscr A := \{J\subset
	I:J\text{ atom in }\mathcal F_K\},$ $\mathscr A_2 := \mathscr A\setminus
	\mathscr A_1$.

	Then, \eqref{eq:lem2-3}(a) is satisfied by the corresponding
	property of each
	$g_\ell$.
	\eqref{eq:lem2-3}(b) follows from the calculation
	\begin{align*}
		\Big| \bigcup_{J\in\mathscr A_1} J \cap V\Big| & \geq
		(1-\tilde{\varepsilon})\sum_{\ell\in\Gamma}|A_{\ell}\cap V|
		\geq (1-\tilde{\varepsilon})^2 |L\cap V| \\
		&\geq (1-\tilde{\varepsilon})^3 |I\cap V|= 
		(1-\varepsilon)|I\cap V|.
	\end{align*}
	Property \eqref{eq:lem2-3}(c) on $L$ is a consequence of property
	\eqref{eq:lem2-3}(c) for the functions $g_\ell$. We can write $\alpha_j
	\equiv \sum_\ell \alpha_{j,\ell} N_\ell$ for some non-negative
	coefficients $(\alpha_{j,\ell})$ that have the property $\sum_{j=1}^M
	\alpha_{j,\ell} =1$ for each $\ell$. Therefore, on $R$, we have
	\begin{equation*}
		\bar x(t) + g(t) = \sum_{j=1}^M \alpha_j(t)x_j +
		\sum_{i=0}^{k-1}  N_{m_i}(t) w_i = \sum_\ell N_\ell(t)\Big( \sum_{j=1}^M \alpha_{j,\ell}
		x_j + \sum_{i=0}^{k-1} \delta_{\ell,m_i}w_i\Big) ,
	\end{equation*}
	which, since $\|w\|_\infty \leq \tilde{\varepsilon}\leq \varepsilon$ and
	$\sum_{j=1}^M \alpha_{j,\ell} = 1$ for each $\ell$, implies
	\eqref{eq:lem2-3}(c) on $R$.
\end{proof}

We now use Lemma \ref{lem:moments} inductively to prove Theorem
\ref{thm:general}.
\begin{proof}[Proof of Theorem \ref{thm:general}]
	We assume that $X$ does not have the RNP. Then,
	by Theorem \ref{thm:dentable},
	the ball $B(0,1/2)\subset X$ contains a 
	non-dentable convex set $D$ 
	satisfying
	\begin{equation*}
		x\in \overline{\operatorname{conv}}\big( D \setminus B(x,2\delta)
		\big),\qquad x\in D
	\end{equation*}
	for some parameter $2\delta$. 
	Defining $D_0 = D+
	B(0,\delta/2)$ and, for $j\geq 1$, $D_j = D_{j-1} +
	B(0,2^{-j-1}\delta)$, we use the remark after Definition
	\ref{def:dentable} to get that all the sets $(D_j)$ are contained in
	$B(0,1)$, are convex and
	\begin{equation*}
		x\in \operatorname{conv}\big(D_j \setminus B(x,\delta)\big),\qquad
		x\in D_j,\ j\geq 0.
	\end{equation*}
	We will assume without restriction that
	$\eta\leq \delta$.
	
	
	Let $x_{0,1}\in D_0$ arbitrary and set $f_0 \equiv \charfun_{[0,1]}
	\otimes x_{0,1}\in
	S_{m_0}\otimes X$ on $I_{0,1} := [0,1]$ for $m_0 = 0$. By
	$P_j$, we will denote the $L^1_X$-extension of the orthogonal projection operator onto $S_{m_j}$,
	where we assume that $(m_j)_{j=1}^n$ and  $(f_j)_{j=1}^n$ with $f_j \in
	S_{m_j}\otimes X$
	for each $j=1,\ldots,n$  are constructed in 
	such a way that for all $j=0,\ldots,n$,
\begin{enumerate}
	\item $P_{j-1} f_j = f_{j-1}$ if $j\geq 1$,
	\item on all atoms $I$ in $\mathcal F_{m_j}$, $f_j$ has the form
		\begin{equation*}
			f_j \equiv \sum_{\ell} \lambda_\ell \otimes y_\ell,\qquad
			\text{finite sum}
		\end{equation*}
		for functions $\lambda_\ell\in S_{m_j}$ with non-negative
		B-spline coefficients, 
		$\sum_\ell \lambda_\ell\equiv 1$  and some
		$y_\ell\in D_j$,
	\item there exists a finite collection of disjoint intervals
		$(I_{j,i})_i$
		that are atoms in $\mathcal F_{m_j}$ so that (setting $C_j = \cup_i
		I_{j,i}$)
		\begin{enumerate}
			\item for all $i$, $f_j \equiv x_{j,i}\in D_j$ on
				$I_{j,i}$,
			\item $\|f_j - f_{j-1}\|_X \geq \delta$ on $C_j\cap C_{j-1}$
				if
				$j\geq 1$,
			\item $|C_j\cap C_{j-1}\cap V| \geq
				(1
				-2^{-j}\eta)|V|
				$ if $j\geq 1$,
			\item $|C_j \cap V| \geq
				(1
				-2^{-j-2}\eta)|V|
				$,
			\item $|I_{j,i}\cap V|>0$ for every $i$.
		\end{enumerate}
\end{enumerate}
We will then perform the construction of $m_{n+1}$, $f_{n+1}$ and the collection
$(I_{n+1,i})$ of atoms in $\mathcal F_{m_{n+1}}$ having
 properties  (1)--(3) for $j=n+1$.
Define the collection $\mathscr C = \{A \text{ atom of }\mathcal F_{m_{n}} : |A\cap
V|>0\}$. We will distinguish the two cases  $B\in \mathscr C_1 := \{A\in
	\mathscr C: A= I_{n,i}\text{ for some }i\}$ and $B\in \mathscr
	C_2:=\mathscr C\setminus \mathscr C_1$.

	\textsc{Case 1: }$B\in \mathscr{C}_1$: here, $B=I_{n,i}$ for some $i$ and
we use the fact that on $B$, $f_n = x_B:=x_{n,i}\in D_n$ and write
\begin{equation*}
	x_{B} = \sum_{\ell=1}^{M_B} \alpha_{B,\ell} x_{B,\ell}
\end{equation*}
with some positive numbers $(\alpha_{B,\ell})$ satisfying $\sum_\ell \alpha_{B,\ell} = 1$, 
some $x_{B,\ell}\in D_n$ and $\|x_{B} - x_{B,\ell}\|\geq \delta$ for any 
$\ell=1,\ldots, M_B$.
We apply Lemma \ref{lem:moments} 
to the interval
$B$ with this decomposition and with the parameter
$\varepsilon=\eta_n:= 2^{-n-3}\eta$.
This yields a
function $g_{B}\in S_{K_{B}}\otimes X$ for some positive integer $K_{B}$ that has the properties

	\begin{enumerate}[(i)]
		\item $\int t^\ell g_{B}(t)\dif t = 0, \qquad 0\leq \ell\leq
			k-1$,
		\item $\supp g_{B} \subset \operatorname{int} B$,
		\item we have a  splitting of the collection $\mathscr A_B =
			\{A\subset B: A \text{ is atom in }\mathcal F_{K_B}\}$ into
			$\mathscr A_{B,1} \cup \mathscr A_{B,2}$  so that  

			\begin{enumerate}[(a)]
				\item on each $J\in \mathscr A_{B,1}$,
					$f_n+g_B=x_B +
					g_B$
					is constant on $J$ taking values in
					$\cup_\ell \{ x_{B,\ell} \}$,
				\item $|\cup_{J\in\mathscr A_{B,1}} J\cap V| \geq
					(1-\eta_n)|B\cap V|$,
				\item on each $J\in \mathscr A_{B,2}$, the
					function $f_n + g_B$ can be written as 
					\begin{equation*}
						f_n(t) + g_B(t) = x_B + g_B(t) =
						\sum_\ell \lambda_{B,\ell}(t)
						y_{B,\ell}
					\end{equation*}
					for some
					functions $\lambda_{B,\ell} \in
					S_{K_B}$ having non-negative B-spline
					coefficients with $\sum_\ell
					\lambda_{B,\ell}\equiv 1$ and
					$y_{B,\ell}\in
					\operatorname{conv}\{x_{B,\ell}: 1\leq
					j\leq M_B\} +
					B(0,\eta_n)$.

			\end{enumerate}
	\end{enumerate}

	\textsc{Case 2: } $B\in \mathscr C_2$:
	 on $B$, $f_n$ is of the form
	\begin{equation*}
		f_n(t) = \sum_{\ell=1}^{M_B} \lambda_{\ell}(t)y_{\ell}
	\end{equation*}
	for some functions $\lambda_{\ell}\in S_{m_n}$ having non-negative
	B-spline coefficients with
	$\sum_\ell \lambda_{\ell}\equiv 1$ and some $y_{\ell}\in D_n$.
	Applying Lemma \ref{lem:moments} 
	with the parameter $\eta_n =2^{-n-3}\eta$, 
	we obtain a function $g_B\in S_{K_B}\otimes X$ (for some
	positive integer $K_B$) that has the properties 
	\begin{enumerate}[(i)]
		\item $\int t^\ell g_{B}(t)\dif t = 0, \qquad 0\leq \ell\leq
			k-1$,
		\item $\supp g_{B} \subset \operatorname{int} B$,

		\item we have a  splitting of the collection $\mathscr A_B =
			\{A\subset B: A \text{ is atom in }\mathcal F_{K_B}\}$ into
			$\mathscr A_{B,1} \cup \mathscr A_{B,2}$  so that  

			\begin{enumerate}[(a)]
				\item for each $J\in \mathscr A_{B,1}$, $f_n +
					g_B$
					is constant on $J$ taking values in
					$\operatorname{conv}\{y_\ell : 1\leq
					\ell\leq M_B\}$,
				\item $|\cup_{J\in\mathscr A_{B,1}} J\cap V| \geq
					(1-\eta_n)|B\cap V|$,
				\item for each $J\in \mathscr A_{B,2}$, the
					function $f_n + g_B$ can be written as 
					\begin{equation*}
						f_n(t) + g_B(t) = 
						\sum_\ell \lambda_{B,\ell}(t)
						y_{B,\ell} 
					\end{equation*}
					for some
					functions $\lambda_{B,\ell} \in
					S_{K_B}$ having non-negative B-spline
					coefficients with $\sum_\ell
					\lambda_{B,\ell}\equiv 1$  and
					$y_{B,\ell}\in
					\operatorname{conv}\{y_j: 1\leq
					j\leq M_B\} +
					B(0,\eta_n)$.

			\end{enumerate}
	\end{enumerate}

	Having treated those two cases, 
	we define the index $m_{n+1}:=\max\{ K_B : B\in\mathscr C\}$
	and
	\begin{equation*}
		f_{n+1} = f_n + \sum_{B\in\mathscr{C}} g_{B}.
	\end{equation*}
	The new collection $(I_{n+1,i})$ is defined to be the decomposition of
	the set $\cup_{B\in\mathscr C} \cup_{J\in \mathscr A_{B,1}} J$ (from the above
	construction)
	into $\mathcal
	F_{m_{n+1}}$-atoms, after deleting those  $\mathcal F_{m_{n+1}}$-atoms $I$
	with $|I\cap V|=0$.
	Since $D_n$ is convex and $\eta\leq \delta$, the corresponding function values of $f_{n+1}$
	are contained in $D_n +
	B(0,\eta_n)\subset D_{n+1}$ and we will  enumerate them as $(x_{n+1,i})_i$
	accordingly. We additionally set $C_{n+1} := \cup_i I_{n+1,i}$.
	
	With these
	definitions, we will successively show properties (1)--(3) for $j=n+1$.
	Since the function $g=P_n f_{n+1} \in S_{m_n}\otimes X$ is characterized by the
	condition
	\begin{equation*}
		\int g(t)s(t)\dif t = \int f_{n+1}(t)s(t)\dif t,\qquad s\in
		S_{m_n},
	\end{equation*}
	property (1) for $j=n+1$ follows if we show that $\int
	g_{B}(t)s(t)\dif t = 0$ for any
	$s\in S_{m_n}$ and any $B\in\mathscr C$. But this is a consequence of
	(i) for $g_{B}$ (in both case 1 and case 2),
	since $s\in S_{m_n}$ is a polynomial of order $k$ on $B$.

	Property (2) now is a consequence of (iii) (again for both cases 1 and
	2), we just remark once again that $D_{n}+ B(0,\eta_n) \subset D_{n+1}$
	due to  $\eta\leq \delta$.
	Properties (3a), (3b) and (3e) are direct consequences of the
	construction. (3d) follows from (iii)(b) in cases 1 and 2 since
	\begin{align*}
		|C_{n+1} \cap V| &=\Big| \bigcup_{B\in\mathscr C} \bigcup_{J\in
			\mathscr A_{B,1}} J\cap V\Big| 
			= \sum_{B\in\mathscr C} \Big| \bigcup_{J\in \mathscr
				A_{B,1}} J\cap V\Big| \\
				&\geq (1-\eta_n) \sum_{B\in \mathscr C} |B\cap
				V| = (1-\eta_n)|V|
	\end{align*}
	and $\eta_n
	= 2^{-n-3}\eta$. 
	For property (3c), we calculate
	\begin{equation*}
		|C_{n+1} \cap C_n \cap V| \geq (1-\eta_n) |C_n\cap V| \geq
		(1-\eta_n) (1-2^{-n-2}\eta) |V|
	\end{equation*}
	by (iii)(b) in case 1 and by induction hypothesis. Since $\eta_n =
	2^{-n-3}\eta$, we get $(1-\eta_n)(1-2^{-n-2}\eta)\geq 1-2^{-(n+1)}\eta$
	and this proves (3c) for $j=n+1$.

	
	Finally, we note that due to  (2) and (3)(c),
	the sequence $(m_n)$, the $k$-martingale spline sequence $(f_n)$ and the sets
	$E_n:= C_n \cap C_{n-1}\cap V$ have the properties that are desired in the
	theorem. 
\end{proof}
\subsection*{Acknowledgments}
The author is grateful to P. F. X. Müller for many fruitful discussions 
related to
the underlying article. This research is supported by the FWF-project
Nr.P27723.


\bibliographystyle{plain}
\bibliography{radon}
\end{document}